\documentclass[a4paper,12pt,reqno]{amsart}
\usepackage{amsmath,amssymb,amsthm}
\usepackage{amsmath}
\usepackage{mathrsfs}
\usepackage{enumerate}
\usepackage{graphicx}
\usepackage{mathrsfs}
\numberwithin{equation}{section}
\theoremstyle{plain}
\newtheorem{thm}{Theorem}[section]

\newtheorem{Lemma}[thm]{Lemma}
\newtheorem{corol}[thm]{Corollary}

\usepackage{mathrsfs}

\newtheorem{rem}[thm]{Remark}

\theoremstyle{definition}
\newtheorem{defn}[thm]{Definition}

\begin{document}
\title[Besov-Hankel norms in terms of the Continuous  Bessel wavelet transform]{ Besov-Hankel norms in terms of the Continuous  Bessel wavelet transform}
\author[Ashish Pathak and Dileep Kumar]{Ashish Pathak and Dileep Kumar  \\
 Department of Mathematics \\
 Institute of Science \\
 Banaras Hindu University \\
    Varanasi-221005, India.}
\date{}
\thanks{$^{1}$ E-mail: ashishpathak@bhu.ac.in,   $^{2}$ E-mail: dkbhu07@gmail.com }
\begin{abstract}
  	In this paper, we extend the concept of continuous Bessel wavelet transform in $L^p$-space and derived the Parseval's as well as the inversion formulas. By using Bessel wavelet coefficients we characterized the Besov- Hankel space.
  \end{abstract}
\maketitle
\section{Introduction}
In this paper, as usual $L^{p,\sigma}(\mathbb{R}^+=(0, \infty))$ denotes the weighted $L^p -$ space with norm 
\begin{eqnarray}
\label{dkp6: eq 1.1}
\| f \|_{L^{p,\sigma}}= \| f \|_{p,\sigma}=\left( \int_{0}^{\infty} |f(x)|^p  d \sigma (x) \right)^{\frac{1}{p}} , (1 \leq p < \infty) , 
\end{eqnarray}
\begin{eqnarray}
\label{dkp6: eq 1.2}
\| f \|_{\infty,\sigma} =  \text{ess sup}_{0<x<\infty} |f(x)| < \infty.
\end{eqnarray}
 The Hankel transformation of the function $f \in L^{1,\sigma} (\mathbb{R}^+)$ is defined by 
 \begin{eqnarray}
 \label{dkp6: eq 1.3}
 	 \hat{f} (x)= \int\limits_{0}^{\infty} j(xt) f(t) d\sigma(t) ,\,\,\,\,  0 \leq x < \infty, 
 \end{eqnarray}
 where $ \sigma(t) = \frac{t^{2\nu +1}}{2^{\nu + \frac{1}{2}} \varGamma(\nu + \frac{3}{2} )}$,   $j(.) = C_\nu x^{\frac{1}{2} - \nu} J_{\nu - \frac{1}{2}}$, $ C_\nu = 2^{\nu + \frac{1}{2}} \varGamma(\nu + \frac{1}{2} )$ and  $J_{\nu - \frac{1}{2}}$ denote the Bessel function of first kind of order $ \nu - \frac{1}{2} $. \\
If  $\hat{f} \in L^{1,\sigma} (\mathbb{R}^+)$, then the inverse of Hankel transformations is given by 
\begin{eqnarray}
\label{dkp6: eq 1.4}
f(x) = \int\limits_{0}^{\infty} j(xt) \hat{f}(t) d\sigma(t) ,\,\,\,\,  0 < x < \infty .
\end{eqnarray}
Also, Parseval's formula of the Hankel transformation for $f, g \in  L^{1,\sigma} \bigcap  L^{2,\sigma} $  is given by 
\begin{eqnarray}
\label{dkp6: eq 1.5}
\int\limits_{0}^{\infty} \hat{f}(x) \hat{g}(x) d\sigma (x) =\int\limits_{0}^{\infty} f(u) g(u) d\sigma (u).
\end{eqnarray}
By denseness and continuity the Parseval's formula can be extended to all $ f, g \in L^{2,\sigma}(\mathbb{R}^+)$. Hence Hankel transform  is isometry on $L^{2,\sigma} (\mathbb{R}^+)$. \\ \\
If $f, g \in  L^{1,\sigma} (\mathbb{R}^+) $, then the convolution associated with the Hankel is defined as  (see \cite{hai1965}) 
\begin{eqnarray}
\label{dkp6: eq 1.6}
(f\# g)(x)= \int\limits_{0}^{\infty}  f(x,y) g(y) d\sigma (y) ,
\end{eqnarray}
  where the Hankel  translation is  given by
\begin{eqnarray}
\label{dkp6: eq 1.7}
 f(x,y) =\tau_y f(x) =: \int\limits_{0}^{\infty} f(z) D(x,y,z)  d\sigma (z),\,\,\, 0 < x,y < \infty, 
\end{eqnarray}
and 
\begin{eqnarray}
\label{dkp6: eq 1.8}
D(x,y,z) &=& \int\limits_{0}^{\infty} j(xu) j(yu) j(zu) d\sigma (u) \nonumber \\ &=& 2^{3\nu - \frac{5}{2}} [\varGamma(\nu + \frac{1}{2} )]^2 \left( \varGamma(\nu) \pi^\frac{1}{2}\right)^{-1} (xyz)^{-2\nu -1} [\Delta(xyz)]^{2 \nu -2}
\end{eqnarray}
where $\Delta $ denotes the area of a triangle.  $  D(x,y,z) $ is symmetric in $x,y,z $. \\
From (1.4) and (\ref{dkp6: eq 1.8}), we have 
\begin{eqnarray}
\label{dkp6: eq 1.9}
\int\limits_{0}^{\infty} j(zu) D(x,y,z) d\sigma (z)  =  j(xu) j(yu),  \,\, 0 < x,y < \infty, \,\, 0 \leq u < \infty,
\end{eqnarray}
and for $u=0$, we get
\begin{eqnarray}
\label{dkp6: eq 1.10}
\int\limits_{0}^{\infty}  D(x,y,z) d\sigma (z)  = 1,
\end{eqnarray} 
and 
\begin{eqnarray}
\label{dkp6: eq 1.11}
\hat{(f\# g)} (x)=\hat{f}(x) \hat{g}(x), \,\, 0 \leq x < \infty.
\end{eqnarray} 
Now, we recall some  properties of Hankel convolution ( see \cite{hi1960},\cite{rspathak2003},\cite{bjrm1998}, \cite{upsi2015} ) which are useful throught the paper.
\begin{Lemma}
\label{dkp6:lemma 1.1}
	Let $f \in  L^{p,\sigma}(\mathbb{R}^+),  \,\, 1 \leq p < \infty$. Then we have
	\begin{eqnarray}
	\label{dkp6: eq 1.12}
	||\tau_y f(x)||_{p,\sigma}\leq ||f||_{p,\sigma}  .
	\end{eqnarray}	 
\end{Lemma}
\begin{Lemma}
\label{dkp6:lemma 1.5}
	Let $f  \in L^{p,\sigma}(\mathbb{R}^+)$ and $g  \in L^{q,\sigma}(\mathbb{R}^+), \,\, \frac{1}{p}+\frac{1}{q}=1+\frac{1}{r} $. Then we have
	\begin{eqnarray}
	\label{dkp6: eq 1.17}
	||f\#  g||_{r,\sigma} \leq ||f||_{p,\sigma} ||g||_{p,\sigma} .
	\end{eqnarray}	 
\end{Lemma}
Betancor and Rodriguez-Mesa defined the new function spaces called  Besov-Hankel by using Hankel transform and its properties  (see \cite{bjrm1998}) 
\begin{defn} \textbf{(Besov-Hankel Space):}
	Let measurable function $ \phi$ on $(0, \infty) $ belongs to $BH^{p,q} _{\alpha, \sigma}$ if  $ \phi \in L^{p,\sigma}(\mathbb{R}^+) $ and 
	\begin{eqnarray}
	\label{dkp6: eq 1.21}
	\int\limits_{0}^{\infty} \left( h^{-\alpha} w_{p} (\phi) (h) \right)^q \frac{dh}{h} < \infty \,\,\,\, \text{for}\,\, \alpha > 0, 1 \leq p,q < \infty,
	\end{eqnarray}
	where $ w_{p} (\phi) (h) =: \lVert \tau_h \phi -  \phi\rVert_{p,\sigma} $,\,\, $ h \in (0, \infty) $. 
\end{defn}
\subsection{Bessel Wavelet }
Using the properites of Hankel transform Pathak and Dixit (see \cite{rspathak2003}) define  the continuous Bessel wavelet for  $\psi \in L^{p,\sigma}(\mathbb{R}^+)$ , $1 \leq p < \infty $,  $ b \geq 0 $ and $ a > 0 $ as \\
\begin{eqnarray}
\label{dkp6: eq 1.18}
\psi_{b,a}(x):& =&D_a \tau_b \psi(x) \nonumber \\ &=& a^{-2\nu-1}\int\limits_{0}^{\infty} \psi(z) D \left( \frac{b}{a},\frac{x}{a},z \right)   d \sigma(z)  
\end{eqnarray} 
where $ D_a$ denote the dilation operator. \\
The continuous Bessel wavelet transform of $f \in  L^{2,\sigma}(\mathbb{R}^+) $ with respect to a wavelet $ \psi \in  L^{2,\sigma}(\mathbb{R}^+)$ is defined as  
%(see \cite{rspathak2003})
\begin{eqnarray}
\label{dkp6: eq 1.19}
(B_\psi f )(b,a) & = & \int\limits_{0}^{\infty} f(x) \overline{\psi_{b,a}}(x) d\sigma (x) \nonumber \\ &=& a^{-2\nu-1} \int\limits_{0}^{\infty} \int\limits_{0}^{\infty} f(x) \overline{\psi(z)} D \left( \frac{b}{a},\frac{x}{a},z\right) d\sigma(z) d\sigma(x).
\end{eqnarray}
Moreover , using (\ref{dkp6: eq 1.6}), we have 
\begin{eqnarray}
\label{dkp6: eq 1.20}
(B_\psi f )(b,a) = (f\# \psi_a)(b),
\end{eqnarray}
where $ \psi_a (t)= a^{-2\nu-1} \overline{\psi(t/a)} $.\\

 Betancor and Rodriguez-Mesa \cite{bjrm1998} first time introduce Besov-Hankel spaces and  characterized by mean of  the Bochner-Riesz mean and the partial Hankel integrals. Perrier and Basdevant \cite{peba1996} established the characterization of Besov spaces by means of continuous wavelet transform. Motivated by these two works, we characterized Besov-Hankel spaces by the continuous Bessel wavelet transform.  \\
Present paper is organized in following manner: section 1 is introductory, in which we recall some properties of  Hankel transform , Besov-Hankel space and continuous Bessel wavelet transform. Section 2 is related to continuous the continuous Besse wavelet transform in $ L^{p,\sigma}(\mathbb{R}^+)$ . In the  section 3, characterize Besov-Hankel norms in terms of continuous Bessel wavelet transform.
\section{The Continuous Bessel Wavelet Transform in $ L^{p,\sigma}(\mathbb{R}^+)$ }
In this section we extend the concept of Bessel wavelet transform on $L^{p,\sigma}(\mathbb{R}^+)$.  
\begin{thm}
\label{dkp6:1 2.1}
	Suppose that the Bessel wavelet $ \psi \in L^{p,\sigma}(\mathbb{R}^+)$ satisfies the admissibility condition 
	\begin{eqnarray*}
		A_\psi = \int ^ \infty _0 \omega^{-2\nu -1} | \hat{\psi}(\omega)|^2d\omega >0,
	\end{eqnarray*}
	where $ \hat{\psi}$ denote the Hankel transform of $ \psi$ then continuous Bessel wavelet transform is a bounded linear operator 
	\begin{eqnarray*}
		L^{p,\sigma}(\mathbb{R}^+)\rightarrow  L^{2,\sigma}(\mathbb{R}^+, \frac{d\sigma(a)}{a^{2\nu +1}}) \times L^{p,\sigma}(\mathbb{R}^+),
	\end{eqnarray*}
	moreover, for any $ f\in L^{p,\sigma}(\mathbb{R}^+) $, $1<p<\infty$
	\begin{eqnarray}
	\label{dkp6:eq 2.1}
	\| f\|_{L^{p,\sigma}(\mathbb{R}^+)} = \left(\int ^ \infty _0 \left( \int ^ \infty _0 |B_\psi f(b,a)|^2\frac{d\sigma(a)}{a^{2\nu +1}} \right)^{\frac{p}{2}} d\sigma(b) \right)^\frac{1}{p}. 
	\end{eqnarray}	
\end{thm}
\begin{proof}
	Let $ S_p$ denote the space $ L^{2,\sigma} (\mathbb{R}^+, \frac{d\sigma(a)}{ a^{2\nu+1}}  ) \times L^{p,\sigma}(\mathbb{R}^+)$ associated to the norm \\
	\begin{center}
		$ \| f\|_{S_p} = \left\lbrace  \int ^ \infty _0 \left(  \int ^ \infty _0 | f(b,a)|^2 \frac{d\sigma(a)}{ a^{2\nu+1}} \right)^\frac{p}{2} d\sigma(b) \right\rbrace^\frac{1}{p} $.
	\end{center}
	If we take $ p=2$, then from Plancherel's theorem:   
	\begin{eqnarray*}
		\| B_\psi f\|_{S_2} &=& \left\lbrace  \int ^ \infty _0 \left(  \int ^ \infty _0 | B_\psi f(b,a)|^2 \frac{d\sigma(a)}{ a^{2\nu+1}} \right) d\sigma(b) \right\rbrace^\frac{1}{2} \\  \| B_\psi f\|_{S_2} &=& \sqrt{A_\psi} \| f\|_{L_{2,\sigma}}, 
	\end{eqnarray*}
	where $A_\psi = \int ^ \infty _0 \omega^{-2\nu-1} | \hat{\psi}(\omega) |^2 d\omega > 0 $, if $ \psi$ is real. From singular integral theorem, the operators on $ L^{2,\sigma} (\mathbb{R}^+, \frac{d\sigma(a)}{ a^{2\nu+1}}  ) $ holds inequality: \\ 
	\begin{center}
		$ \| B_\psi f\|_{S_p} \leq C_p \|f\|_{L^{p,\sigma}(\mathbb{R}^+)} $ \text{for} $ 1<p \leq 2$,
	\end{center}
	where the constant $ C_p$ depends only on $ p$ and $ \psi$(see \cite{ems1970}). Due to duality the inequality is also valid for $ 1<p<\infty$. It follows that 
\begin{eqnarray}
\label{dkp6:eq 2.2}
	\left\lbrace  \int ^ \infty _0 \left(  \int ^ \infty _0 | B_\psi f(b,a)|^2 \frac{d\sigma(a)}{ a^{2\nu+1}} \right)^\frac{p}{2} d\sigma(b) \right\rbrace^\frac{1}{p} \leq C_p \|f\|_{L^{p,\sigma}(\mathbb{R}^+)}.
\end{eqnarray}
	Conversely suppose that $ f \in L^{2,\sigma}(\mathbb{R}^+) \cap L^{p,\sigma}(\mathbb{R}^+)$. Since continuous Bessel wavelet transform is isomerty for every $ g \in L^{2,\sigma}(\mathbb{R}^+) \cap L^{q,\sigma}(\mathbb{R}^+)$, we can write 
	\begin{eqnarray}
	\label{dkp6:eq 2.3}
	\int ^ \infty _0  \int ^ \infty _0  B_\psi f(b,a) \overline{  B_\psi g(b,a)} a^{-2\nu - 1} d\sigma(a) d\sigma(b) &=& A_\psi \langle f, g\rangle \nonumber \\ \frac{1}{A_\psi} \int ^ \infty _0  \int ^ \infty _0  B_\psi f(b,a) \overline{ B_\psi g(b,a)} a^{-2\nu - 1} d\sigma(a) d\sigma(b) &=&  \int ^ \infty _0 f(x) \overline{g(x)} d \sigma (x).
	\end{eqnarray}
	Now,  
	\begin{eqnarray*}
		|\int ^ \infty _0 f(x) g(x) d \sigma (x)| &=& \frac{1}{A_\psi} | \int ^ \infty _0  \int ^ \infty _0  B_\psi f(b,a) \overline{ B_\psi g(b,a)} a^{-2\nu - 1} d\sigma(a) d\sigma(b)| \\ &\leq& \frac{1}{A_\psi} \int ^ \infty _0  \int ^ \infty _0 |  B_\psi f(b,a) \overline{ B_\psi g(b,a)}| a^{-2\nu - 1} d\sigma(a) d\sigma(b),
	\end{eqnarray*}
	using  Schwarz inequality and then Holder's inequality, we have
	\begin{eqnarray*}
		& \leq & \frac{1}{A_\psi} \left( \int ^ \infty _0 \left( \int ^ \infty _0 |B_\psi f(b,a)|^2 a^{-2\nu - 1} d\sigma(a)\right)^\frac{p}{2} d\sigma(b)  \right)^\frac{1}{p} \\ && \,\,\,\,\,\,\,\,\,\,\,\,\,\,\,\,  \times \left( \int ^ \infty _0 \left( \int ^ \infty _0 |B_\psi g(b,a)|^2 a^{-2\nu - 1} d\sigma(a)\right)^\frac{q}{2} d\sigma(b)  \right)^\frac{1}{q},
	\end{eqnarray*}
	where $ \frac{1}{p} + \frac{1}{q} = 1$.\\ From equation (\ref{dkp6:eq 2.2}), we get
	\begin{eqnarray*}
		& \leq & \frac{A_q}{A_\psi} \left( \int ^ \infty _0 \left( \int ^ \infty _0 |B_\psi f(b,a)|^2 a^{-2\nu - 1} d\sigma(a)\right)^\frac{p}{2} d\sigma(b)  \right)^\frac{1}{p}  \| g\|_{L^{q,\sigma} (\mathbb{R}^+)}.
	\end{eqnarray*}
	By Density theorem 
	\begin{eqnarray*}
		\| f \|_{L^{p,\sigma} (\mathbb{R}^+)}	& \leq & A \left( \int ^ \infty _0 \left( \int ^ \infty _0 |B_\psi f(b,a)|^2 a^{-2\nu - 1} d\sigma(a)\right)^\frac{p}{2} d\sigma(b)  \right)^\frac{1}{p},  
	\end{eqnarray*}
	where $A = \frac{A_q}{A_\psi} $.
\end{proof}
\begin{thm} (\textbf{Parseval's formula})
\label{dkp6:1 2.2}
	Let us assume $  \phi_1 \in L^{p,\sigma}(\mathbb{R^+})$, $ \phi_2 \in L^{q,\sigma}(\mathbb{R^+})$ with $ 1 \leq p,q < \infty $ and $ \frac{1}{p} + \frac{1}{q} = 1$. If $ \psi$ is a real wavelet then
	\begin{eqnarray}
	\label{dkp6:eq 2.4}
	\frac{1}{A_\psi} \int ^ \infty _0  \int ^ \infty _0  B_\psi \phi_1 (b,a) \overline{ B_\psi \phi_2 (b,a)} a^{-2\nu - 1} d\sigma(a) d\sigma(b) &=&  \int ^ \infty _0 \phi_1(x) \overline{\phi_2(x)} d\sigma(x), \nonumber \\ 
	\end{eqnarray}
	where $ A_\psi = \int ^ \infty _0 \omega^{-2\nu-1} | \hat{\psi}(\omega) |^2 d\omega > 0 $ and $  \hat{\psi} $ denote the Hankel transform.
\end{thm}
\begin{proof}
	Let us define bilinear transform $ T : L^{p,\sigma}(\mathbb{R}^+) \times L^{q,\sigma}(\mathbb{R}^+) \rightarrow \mathbb{R}^+ $ by \\   
	\begin{center}
		$ T(\phi_1, \phi_2) = \langle B_\psi \phi_1 (b,a), B_\psi \phi_2 (b,a) \rangle _{(\frac{d\sigma(a)}{a^{2\nu + 1}}d\sigma(b))}$. 
	\end{center} 
	Now, applying Holder's inequality two times we obtain
	\begin{eqnarray*}
		\lvert T(\phi_1, \phi_2) \rvert && = \lvert \langle B_\psi \phi_1 (b,a), B_\psi \phi_2 (b,a) \rangle _{\frac{d\sigma(a)}{a^{2\nu + 1}}, d\sigma(b)} \rvert \\ && \leq \int ^ \infty _0 \left( \int ^ \infty _0 \lvert B_\psi \phi_1 (b,a) \rvert^2 \frac{d\sigma(a)}{a^{2\nu + 1}} \right)^\frac{1}{2} \left( \int ^ \infty _0 \lvert B_\psi \phi_2 (b,a) \rvert^2 d\sigma(b) \right)^\frac{1}{2} \\ && \leq \left( \int ^ \infty _0 \left( \int ^ \infty _0 \lvert B_\psi \phi_1 (b,a) \rvert^2 \frac{d\sigma(a)}{a^{2\nu + 1}} \right)^\frac{p}{2} \right)^\frac{1}{p} \times  \left( \int ^ \infty _0 \left( \int ^ \infty _0 \lvert B_\psi \phi_2 (b,a) \rvert^2 d\sigma(b) \right)^\frac{q}{2}\right)^\frac{1}{q}  
	\end{eqnarray*}
	using Theorem 2.1. we have 
	\begin{eqnarray}
	\label{dkp6:eq 2.5}
	\lvert T(\phi_1, \phi_2) \rvert \leq C \lVert \phi_1 \rVert_{L^{p,\sigma}(\mathbb{R_+})} \lVert \phi_2\rVert_{L^{q,\sigma}(\mathbb{R_+})}.
	\end{eqnarray}
	Moreover for all $ \phi_1 \in L^{2,\sigma}(\mathbb{R}^+)\cap L^{p,\sigma}(\mathbb{R}^+) $ and $ \phi_2 \in L^{2,\sigma}(\mathbb{R}^+)\cap L^{q,\sigma}(\mathbb{R}^+) $ we get
	\begin{eqnarray}
	\label{dkp6:eq 2.6}
	T(\phi_1, \phi_2) = \langle B_\psi \phi_1 (b,a), B_\psi \phi_2 (b,a) \rangle _{\frac{d\sigma(a)}{a^{2\nu + 1}}, d\sigma(b)} = A_\psi \langle \phi_1, \phi_2 \rangle.
	\end{eqnarray}
	From equations (3.5 ), (3.6) and density of spaces $ L^{2,\sigma}(\mathbb{R}^+)\cap L^{p,\sigma}(\mathbb{R}^+) $ in $ L^{p,\sigma}(\mathbb{R}^+) $ gives the result.
\end{proof}	
\subsection{An inversion formula}
\begin{thm}
\label{dkp6:1 2.3}
	Let us consider $\phi \in L^{p,\sigma}(\mathbb{R^+})$ with $ 1 < p < \infty$ and $ \psi$ is a real wavelet. Then
	\begin{eqnarray}
	\label{dkp6:eq 2.7}
	\phi(x) = \frac{1}{A_\psi} \int ^ \infty _0  \int ^ \infty _0 B_\psi \phi (b,a) \psi_{b,a}(x) \frac{d\sigma(a)}{a^{2\nu + 1}} d\sigma(b).
	\end{eqnarray}
	The equality holds in $L^{p,\sigma}(\mathbb{R}^+)$ sense and the integral of right hand side have to be taken in the sense of distributions.
\end{thm}
\begin{proof}
		The proof followed from Theorem 2.2.
\end{proof}	
\section{Characterization of Besov-Hankel Norms}
In present section, By using the above results, we characterize the Besov-Hankel norms associated  the Bessel wavelet transform. 
\begin{thm}
\label{dkp6:1 2.4}
	Let $ f \in B^{p,q}_{\alpha, \sigma} (\mathbb{R}^+) \left(p,q>1 ,  \alpha \neq \mathbb{Z}  \right) $ and analysing wavelet $ \psi $ has $ [\alpha] +1$ cancellations and $ (z^{\alpha - [\alpha]} \psi) \in L^{1,\sigma} (\mathbb{R}^+)$, then the wavelet coefficient of function $ f$ holds following  conditions:  \\ 
	\begin{center}
		if $ q< \infty$, \,\,\,	$   \int ^ \infty _0 \left[ a^{-\alpha} \| B_{\psi} (.,a) \|_{L^{p, \sigma}}  \right]^q \frac{da}{ a} < \infty$ \\
		if $ q = \infty$, \,\,\,	$ a\rightarrow a^{-\alpha} \| B_{\psi} (.,a)  \|_{L^{p, \sigma}} \in L^\infty (\mathbb{R}^+)$.
	\end{center}
	Moreover the function 	$ a\rightarrow a^{-\alpha} \| B_{\psi} (.,a)  \|_{L^{p, \sigma}} \in L^q (\mathbb{R}^+, \frac{da}{ a^{2\nu + 1}})$ and we have: 
\begin{eqnarray}
	\label{dkp6:eq 2.9}
	\Arrowvert a^{-\alpha} \| B_{\psi} (.,a)  \|_{L^{p, \sigma}} \Arrowvert _{L^{q,\sigma} ( \frac{da}{a}) } \leq \|z^{\alpha - [\alpha]} \psi \|^{L_{1,\sigma} } \times \|h^{\alpha - [\alpha]} \omega_p(f,h)\|_{L^{q,\sigma}( \frac{dh}{h})}  
\end{eqnarray}
\end{thm}
\begin{proof}
	By the definition of continuous Bessel wavelet transform, we have 
	\begin{eqnarray}
	B_{\psi} (b,a) &=&  \int ^ \infty _0 f(x) \overline{\psi_{b,a}(x)}d\sigma(x) \nonumber \\ &=&  \int ^ \infty _0 f(x)  \left( \int ^ \infty _0 a^{-2\nu -1} D\left( \frac{b}{a}, \frac{x}{a}, z\right) \overline{\psi(z)}d\sigma(z) \right)    d\sigma(x) \nonumber \\ &=& \int ^ \infty _0 \overline{\psi(z)} \left( \int ^ \infty _0 a^{-2\nu -1} D\left( \frac{b}{a}, \frac{x}{a}, z\right)  f(x) d\sigma(x)  \right) d\sigma(z) \nonumber \\ &=& 
%\int ^ \infty _0 \overline{\psi(z)} \left( \int ^ \infty _0 a^{-2\nu -1} \left( \int ^ \infty _0 J(\frac{b}{a} t) J(\frac{x}{a} t) J(z t) f(x) d\sigma(t) \right) d\sigma(x)  \right) d\sigma(z). \nonumber 
%	\end{eqnarray}
%	Let $\frac{t}{a} = u \Rightarrow  d\sigma(t) =  d\sigma(au) \Rightarrow  d\sigma(t) = a^{2\nu}  d\sigma(u)$. Then 
%	\begin{eqnarray}	B_{\psi} (b,a) &=& 
%\int ^ \infty _0 \overline{\psi(z)} \left( \int ^ \infty _0 a^{-2\nu -1} \left( \int ^ \infty _0 a^{2\nu} J(bu) J(xu ) J((az)u) d\sigma(u) \right)  f(x) d\sigma(x)  \right) d\sigma(z). \nonumber \\ &=& 
\int ^ \infty _0 \overline{\psi(z)} \left( \int ^ \infty _0  D\left(b, x, az\right) f(x) d\sigma(x) \right) d\sigma(z) \nonumber \\  &=&  \left\lbrace \int ^ \infty _0 (\tau_{az}f)(b) \overline{\psi(z)} d\sigma(z) - \int ^ \infty _0 f(b) \overline{\psi(z)} d\sigma(z)  \right\rbrace  \nonumber \\  &=&  \int ^ \infty _0  \overline{\psi(z)}   \left(  (\tau_{az}f)(b) -  f(b) \right)  d\sigma(z). \nonumber 
	\end{eqnarray}
	Taking $ L^{p, \sigma} -$ norm of the wavelet coefficient 
	\begin{eqnarray}
	\| B_{\psi} (b,a) \|_{L^{p, \sigma}} = \int ^ \infty _0  \left\lbrace  |\int ^ \infty _0  \overline{\psi(z)}   \left(  (\tau_{az}f)(b) -  f(b) \right)  d\sigma(z)|^p \right\rbrace ^{\frac{1}{p}} d\sigma(b). \nonumber 
	\end{eqnarray}
	Using Minkowski inequality of integrability for $ p \neq \infty$
	\begin{eqnarray}
\label{dkp6:eq 2.10}
	\| B_{\psi} (b,a) \|_{L^{p, \sigma}} \leq \int ^ \infty _0  \left\lbrace  \int ^ \infty _0    |(\tau_{az}f)(b) -  f(b) | d\sigma(b)|^p \right\rbrace ^{\frac{1}{p}} |\psi(z)|  d\sigma(z).
	\end{eqnarray}
	Suppose that $ q < \infty$ and integrating w.r.t. $  a $, we get
	\begin{eqnarray}
	\int ^ \infty _0 \left[ a^{-\alpha} \| B_{\psi} (b,a) \|_{L^{p, \sigma}} \right]^q \frac{da}{a} \leq \int ^ \infty _0 \left[ a^{-\alpha} \int ^ \infty _0 |\psi(z)| \omega _{p} (f, az)  d\sigma(z) \right]^q \frac{da}{a}.  \nonumber 
	\end{eqnarray}
	Again using Minkowski integrabilty  inequality 
	\begin{eqnarray*}
		\int ^ \infty _0 \left[ a^{-\alpha} \| B_{\psi} (b,a) \|_{L^{p, \sigma}}  \right]^q \frac{da}{a} &\leq &\left[ \int ^ \infty _0 |\psi(z)| d\sigma(z) \left\lbrace \int ^ \infty _0 \left(  a^{-\alpha}  \omega _{p} (f, az)  \right) ^q \frac{da}{a} \right\rbrace ^{\frac{1}{q}}\right] ^q.  
	\end{eqnarray*}
	Applying change of variable $ h= az $
	\begin{eqnarray}
	\label{dkp6:eq 2.11}
	%& = &\left[ \int ^ \infty _0  |\psi(z)| d\sigma(z) \left\lbrace \int ^ \infty _0 \left(  (\frac{h}{z})^{-\alpha}  \omega _{p,\sigma} (f, h)  \right) ^q \frac{d\sigma(h)}{ h^{2\nu}}\right\rbrace ^{\frac{1}{q}}\right] ^q  \nonumber \\
	  & = &\left[ \int ^ \infty _0 z^{-\alpha} |\psi(z)| d\sigma(z) \left\lbrace \int ^ \infty _0 \left(  h^{-\alpha}  \omega _{p,\sigma} (f, h)  \right) ^q \frac{dh}{h}\right\rbrace ^{\frac{1}{q}}\right] ^q \nonumber \\  & \leq & \left\lbrace  \int ^ \infty _0 z^{-\alpha} |\psi(z)| d\sigma(z) \right\rbrace ^q \times \left\lbrace \int ^ \infty _0 \left(  h^{-\alpha}  \omega _{p,\sigma} (f, h)  \right) ^q \frac{dh}{h}\right\rbrace \nonumber \\  & = & \left\lbrace  \int ^ \infty _0  | z^{-\alpha} \psi(z)| d\sigma(z) \right\rbrace ^q \times \left\lbrace \int ^ \infty _0 \left(  h^{-\alpha}  \omega _{p,\sigma} (f, h)  \right) ^q \frac{dh}{h}\right\rbrace \nonumber \\  & < & \infty.
	\end{eqnarray}
	If $ q = \infty$ the hypothesis on $ f$ says that $ h^{-\alpha}  \omega _{p,\sigma} (f, h)  \in {L_{\infty,\sigma}(\mathbb{R}^+)} $, so 
	\begin{eqnarray}
	\label{dkp6:eq 2.12}
	\| B_{\psi} (b,a) \|_{L_{p, \sigma}} \leq a^\alpha  \|h^{-\alpha}  \omega _{p,\sigma} (f, h)\| _{L_{\infty,\sigma}(\mathbb{R}^+)} \int ^ \infty _0  | z^{-\alpha} \psi(z)| d\sigma(z).
	\end{eqnarray}
	The theorem has been proved for $ 0 < \alpha < 1 $. For $ 1 < \alpha < 2$, $ \alpha $ is not an integer, by hypothesis $ f'$ belongs to $ BH^{p}_{\alpha -1,\sigma}$. Since $ \psi $ has cancellation up to order 2 therefore ,$ \left( z - \int ^z  _{-\infty}   \psi(t) d\sigma(t) \right) $ is an analysing wavelet and we can apply inequality (\ref{dkp6:eq 2.11}) to the coefficient $ C\left( f', \int \psi, a, b\right) $
\begin{eqnarray}
	\int ^ \infty _0 \left( \| C\left( f', \int \psi, a, . \right) \|_{L_{\infty,\sigma}(\mathbb{R}^+)} \right)^q \frac{da}{a} \leq  \left\lbrace  \int ^ \infty _0  | z^{\alpha-1} \psi(z)| d\sigma(z) \right\rbrace ^q \nonumber \\ \times \left\lbrace \int ^ \infty _0 \left(  h^{-\alpha-1}  \omega _{p,\sigma} (f', h)  \right) ^q \frac{dh}{h}\right\rbrace. \nonumber
\end{eqnarray}
	Now, 
	\begin{eqnarray}
	\int ^ \infty _0 \left( a^{-\alpha } \| B_{\psi} (b,a) \|_{L_{p, \sigma}} \|\right)^q \frac{d\sigma(a)}{ a^{2\nu}} \leq  \left\lbrace  \int ^ \infty _0  | z^{\alpha-1} \psi(z)| d\sigma(z) \right\rbrace ^q \nonumber \\ \,\,\,\,\,\,\,\,\,\,\, \times \left\lbrace \int ^ \infty _0 \left(  h^{-\alpha-1}  \omega _{p,\sigma} (f', h)  \right) ^q \frac{dh}{h}\right\rbrace < \infty. \nonumber
	\end{eqnarray}
	This prove the result for $ 1 < \alpha < 2$. Similarly for $ q= \infty$. Hence by recurrence on $ [\alpha]$ proved the theorem. 
\end{proof}
 Next theorem is the converse of the above theorem. The Bessel wavelet coefficients at small value of a is sufficient to  characterizes Besov-Hankel spaces.
\begin{thm}
\label{dkp6:1 2.5}
	Suppose $ \alpha > 0$, $ \alpha$  not an integer and a function $ \psi $ is a real $ C^{[\alpha]+1}$- regular analysing wavelet with all derivatives rapidly decreasing. If $ f, f', f'', f''', ..., f^[\alpha] \in  L^{p, \sigma}(\mathbb{R^+})$ $(1 < p < \infty)$, and if  $ a^{-\alpha} \| (B_\psi f)(a,\cdot) \| _{L ^{p,\sigma}} \in L^{q, \sigma}(\mathbb{R}^+, \frac{da}{a}) $, then $ f \in  BH^{p,q}_{\alpha, \sigma}$ and we have
	\begin{eqnarray}
	\label{dkp6:eq 2.13}
	\lVert h^{-(\alpha - [\alpha] } w_{p,\sigma} (f^{([\alpha])}, h)\rVert_{L^{q,\sigma}} && \leq \frac{1}{A_\psi} \left( \frac{2}{(\alpha - [\alpha] )} \Arrowvert \psi^{[\alpha]} \Arrowvert_{L^{1,\sigma}} + \frac{1}{1-(\alpha - [\alpha] )} \Arrowvert \psi^{[\alpha] + 1} \Arrowvert_{L^{1,\sigma}} \right) \nonumber \\ && \times a^{-\alpha} \| (B_\psi f)(a,\cdot) \| _{L ^{p,\sigma}}
	\end{eqnarray}
\end{thm}
\begin{proof}
	Let $ f\in L^{p,\sigma}(\mathbb{R}^+)$. By inversion formula of Bessel wavelet transform
	\begin{eqnarray}
\label{dkp6:eq 2.14}
	f(x) = \frac{1}{A_\psi} \int_{0}^{\infty} \frac{d \sigma(a)}{a^{2\nu +1}} \int_{0}^{\infty} (B_\psi f)(a,b) \psi_{a,b} (x) d \sigma(b) 
	\end{eqnarray}
	and
	\begin{eqnarray}
\label{dkp6:eq 2.15}
	\tau_h f(x) = \frac{1}{A_\psi} \int_{0}^{\infty} \frac{d \sigma(a)}{a^{2\nu +1}} \int_{0}^{\infty} (B_\psi f)(a,b) \tau_h \psi_{a,b} (x) d \sigma(b). 
	\end{eqnarray}
	Then 
	\begin{eqnarray*}
		\tau_h f(x) - f(x) &=& \frac{1}{A_\psi} \int_{0}^{\infty} \frac{d \sigma(a)}{a^{2\nu +1}} \int_{0}^{\infty} (B_\psi f)(a,b) \left\lbrace  \tau_h \psi_{a,b}(x) - \psi_{a,b}(x) \right\rbrace d \sigma(b) \\ &=& \frac{1}{A_\psi} \int_{0}^{\infty} \frac{d \sigma(a)}{a^{2\nu +1}} \int_{0}^{\infty} (B_\psi f)(a,b) \left\lbrace  \tau_\frac{h}{a} \tau_\frac{b}{a} \psi (\frac{x}{a}) - \tau_\frac{b}{a} \psi (\frac{x}{a}) \right\rbrace d \sigma(b)  \\ &=& \frac{1}{A_\psi} \int_{0}^{\infty} \frac{d \sigma(a)}{a^{2\nu +1}} \int_{0}^{\infty} (B_\psi f)(a,b) a^{-2\nu -1} D(\frac{b}{a},\frac{x}{a}, y) d \sigma(b) \\ && \,\,\,\,\,\,\,\,\,\,\,\,\,\,\,\,\,\,\,\,\,\,\,\,\,\, \times \int_{0}^{\infty} \left\lbrace  \tau_\frac{h}{a}  \psi (y) -  \psi (y) \right\rbrace d\sigma(y)  \\ &=& \frac{1}{A_\psi} \int_{0}^{\infty} \frac{d \sigma(a)}{a^{2\nu +1}}  \tau_\frac{y}{a} (B_\psi f)(a,\frac{x}{a})   \int_{0}^{\infty} \left\lbrace  \tau_\frac{h}{a}  \psi (y) -  \psi (y) \right\rbrace d\sigma(y) 
	\end{eqnarray*}	
	Taking $L ^{p,\sigma}$- norm on both side, we have
	\begin{eqnarray*}
		w_{p,\sigma} (f,h) &=& \frac{1}{A_\psi} \left\lbrace  \int_{0}^{\infty} \arrowvert \int_{0}^{\infty} \frac{d \sigma(a)}{a^{2\nu +1}}  \tau_\frac{y}{a} (B_\psi f)(a,\frac{x}{a})   \int_{0}^{\infty} \left\lbrace  \tau_\frac{h}{a}  \psi (y) -  \psi (y) \right\rbrace d\sigma(y) \arrowvert ^p d\sigma(x) \right\rbrace ^\frac{1}{p} 
	\end{eqnarray*}	
	applying Minkowski's inequality	
	\begin{eqnarray*}
		& \leq & \frac{1}{A_\psi} \int_{0}^{\infty} \frac{d \sigma(a)}{a^{2\nu +1}}    \int_{0}^{\infty}  \arrowvert \tau_\frac{h}{a}  \psi (y) -  \psi (y)  \arrowvert d\sigma(y) \left\lbrace \int_{0}^{\infty} \lvert \tau_\frac{y}{a} (B_\psi f)(a,\frac{x}{a}) \rvert ^p  d\sigma(x) \right\rbrace ^\frac{1}{p} 
	\end{eqnarray*}
	\begin{eqnarray*}
		& = & \frac{1}{A_\psi} \int_{0}^{\infty} \frac{d \sigma(t)}{t^{2\nu}}  \| (B_\psi f)(\frac{h}{t},\cdot) \| _{L ^{p,\sigma}} \int_{0}^{\infty}  \arrowvert \tau_\frac{h}{a}  \psi (y) -  \psi (y)  \arrowvert d\sigma(y) 
	\end{eqnarray*}
	Now,consider $ 0 < \alpha < 1$ and using Minkowski's inequality	
	\begin{eqnarray*}
		\left\lbrace \int_{0}^{\infty} \frac{dh}{h} h^{-\alpha q} w_{p,\sigma} (f,h)^q \right\rbrace ^\frac{1}{q} 	& \leq &\frac{1}{A_\psi} \int_{0}^{\infty} \frac{d \sigma(t)}{t^{2\nu}}  \int_{0}^{\infty}  \arrowvert \tau_\frac{h}{a}  \psi (y) -  \psi (y)  \arrowvert d\sigma(y)\\ && \times  \left\lbrace \int_{0}^{\infty} \frac{dh}{h} h^{-\alpha q} \| (B_\psi f)(\frac{h}{t},\cdot) \| _{L ^{p,\sigma}}^q \right\rbrace ^\frac{1}{q} 
	\\ 	& =& \frac{1}{A_\psi} \int_{0}^{\infty} \frac{d \sigma(t)}{t^{2\nu+\alpha}}  \int_{0}^{\infty}  \arrowvert \tau_\frac{h}{a}  \psi (y) -  \psi (y)  \arrowvert d\sigma(y)
	\\ & & \times \left\lbrace \int_{0}^{\infty} \frac{da}{a} a^{-\alpha q} \| (B_\psi f)(a,\cdot) \| _{L ^{p,\sigma}}^q \right\rbrace ^\frac{1}{q}  
	\\ 	& =& \frac{C}{A_\psi} \int_{0}^{\infty} \frac{dt}{t^{1+\alpha}}  \int_{0}^{\infty}  \arrowvert \tau_\frac{h}{a}  \psi (y) -  \psi (y)  \arrowvert d\sigma(y) 
	 \\ && \times \left\lbrace \int_{0}^{\infty} \frac{da}{a} a^{-\alpha q} \| (B_\psi f)(a,\cdot) \| _{L ^{p,\sigma}}^q \right\rbrace ^\frac{1}{q}.  
	\end{eqnarray*}
	Using Lemma \ref{dkp6:lemma 1.1}, we obtain
	\begin{eqnarray}
\label{dkp6:eq 2.16}
	\int_{0}^{\infty}  \arrowvert \tau_\frac{h}{a}  \psi (y) -  \psi (y)  \arrowvert d\sigma(y) \leq 2\Arrowvert \psi\Arrowvert_{L^{1,\sigma}}
	\end{eqnarray}
	and
	\begin{eqnarray}
	\label{dkp6:eq 2.17}
	\int_{0}^{\infty}  \arrowvert \tau_\frac{h}{a}  \psi (y) -  \psi (y)  \arrowvert d\sigma(y) 
	%&=& \int_{0}^{\infty} \arrowvert \int_{0}^{t} \frac{d}{dz}\tau_z \psi(y) dz\arrowvert  d\sigma(y) \nonumber \\ 
	&=& \int_{0}^{\infty} \arrowvert \int_{0}^{t} \left( \tau_z \psi(y)\right)'  dz\arrowvert  d\sigma(y) \nonumber \\ &\leq& \int_{0}^{t} \int_{0}^{\infty} \arrowvert\left( \tau_z \psi(y)\right)'  dz\arrowvert  d\sigma(y) \nonumber \\ &\leq&  \lVert \psi'\rVert_{L^{1,\sigma}} t.
	\end{eqnarray}
	Here, we observe that
	\begin{eqnarray*}
		\int_{0}^{\infty} \frac{dt}{t^{1+\alpha}}  \int_{0}^{\infty}  \arrowvert \tau_\frac{h}{a}  \psi (y) -  \psi (y)  \arrowvert d\sigma(y) && = \int_{0}^{1} \frac{dt}{t^{1+\alpha}}  \int_{0}^{\infty}  \arrowvert \tau_\frac{h}{a}  \psi (y) -  \psi (y)  \arrowvert d\sigma(y) \\ && + \int_{1}^{\infty} \frac{dt}{t^{1+\alpha}}  \int_{0}^{\infty}  \arrowvert \tau_\frac{h}{a}  \psi (y) -  \psi (y)  \arrowvert d\sigma(y) \\
	&& \leq 2\Arrowvert \psi\Arrowvert_{L^{1,\sigma}} \int_{1}^{\infty} \frac{dt}{t^{1+\alpha}}  + 2\lVert \psi'\rVert_{L^{1,\sigma}} \int_{0}^{1} \frac{dt}{t^{\alpha}} \\ && = \frac{2}{\alpha} \Arrowvert \psi\Arrowvert_{L^{1,\sigma}} + \frac{1}{1-\alpha} \Arrowvert \psi' \Arrowvert_{L^{1,\sigma}},
	\end{eqnarray*}
	which proved the result for $ 0 < \alpha < 1$.\\ If $ 1 < \alpha < 2$, by the hypothesis $ f' \in L^p(\mathbb{R^+})$ and $ \psi$ is a $C^1$- regular function with $ \psi' $ rapidly decreasing at infinity. From equations (\ref{dkp6:eq 2.14}) and (\ref{dkp6:eq 2.15}),we have the equality
	\begin{eqnarray*}
		\tau_h f'(x) - f'(x) &=& \frac{1}{A_\psi} \int_{0}^{\infty} \frac{d \sigma(a)}{a^{2\nu +1}} \int_{0}^{\infty} (B_\psi f)(a,b) \left\lbrace  \tau_h \psi'_{a,b}(x) - \psi'_{a,b}(x) \right\rbrace d \sigma(b)
	\end{eqnarray*}
	Calculate in sinmilar manner as above for $ f'$ gives the following estimation
	\begin{eqnarray*}
		\left\lbrace \int_{0}^{\infty} \frac{dh}{h} h^{-(\alpha -1 )q} w_{p,\sigma} (f,h)^q \right\rbrace ^\frac{1}{q} & \leq & \frac{1}{A_\psi} \int_{0}^{\infty} \frac{dt}{t^{1+\alpha}}  \int_{0}^{\infty}  \arrowvert \tau_\frac{h}{a}  \psi (y) -  \psi (y)  \arrowvert d\sigma(y) \\ && \times \left\lbrace \int_{0}^{\infty} \frac{dh}{h} (a)^{-\alpha q} \| (B_\psi f)(a,\cdot) \| _{L ^{p,\sigma}}^q \right\rbrace ^\frac{1}{q} \\ & \leq & \frac{1}{A_\psi} \left( \frac{2}{(\alpha-1)} \Arrowvert \psi' \Arrowvert_{L^{1,\sigma}} + \frac{1}{1-(\alpha -1)} \Arrowvert \psi'' \Arrowvert_{L^{1,\sigma}} \right) \\ && \times \left\lbrace \int_{0}^{\infty} \frac{dh}{h} a ^{-\alpha q} \| (B_\psi f)(a,\cdot) \| _{L ^{p,\sigma}}^q \right\rbrace ^\frac{1}{q}
	\end{eqnarray*}
	this proves that $ f' \in BH^{p,q}_{\alpha-1, \sigma}$ the hypothesis $ a \rightarrow a^{-\alpha} \| (B_\psi f)(a,\cdot) \| _{L ^{p,\sigma}}^q  \in L^{q,\sigma}(\mathbb{R^+})$ implies  $ a \rightarrow a^{-(\alpha-1)} \| (B_\psi f)(a,\cdot) \| _{L ^{p,\sigma}}^q  \in L^{q,\sigma}(\mathbb{R^+})$ then $ f \in  BH^{p,q}_{\alpha-1, \sigma} $. The theorem is established for $ 1 < \alpha < 2$, a recurrence on $[\alpha]$ gives the final result. 
\end{proof}	
\begin{corol}
	Let $ f \in B^{p,q}_{\alpha, \sigma} (\mathbb{R}^+) \left(p,q>1 ,  \alpha \neq \mathbb{Z}  \right) $ , then
\begin{eqnarray*}
||f||_{B^{p,q}_{\alpha, \sigma}}=||f||_{L^{p,\sigma} (\mathbb{R}^+)}+|f|_{{B^{p,q}_{\alpha, \sigma}}} 
\end{eqnarray*}
where $|f|_{{B^{p,q}_{\alpha, \sigma}}}$  is equal to 
\begin{eqnarray*}
	|f|^q_{B^{p,q}_{\alpha, \sigma}}= \int\limits_{0}^{\infty} \left( h^{-\alpha} w_{p} (\phi) (h) \right)^q \frac{dh}{h}  \approx \int ^ \infty _0 \left[ a^{-\alpha} \| B_{\psi} (.,a) \|_{L^{p, \sigma}}  \right]^q \frac{da}{ a} .
\end{eqnarray*}
\end{corol}
\begin{rem} Considerable work has already been done on the characterization of Besov-k-Hankel norms by means of k-Hankel wavelet transform and Mehler-Besov-Fock spaces by using  Mehler-Fock wavelet transform etc.. A further research in the context of different types of Besov space related to the different integral transform is needed.
\end{rem}
\section*{acknowledgement}
	The research of the second author is supported by University Grants Commission ( UGC), grant number: F.No. 16-6(DEC. 2017)/2018(NET/CSIR),  New Delhi, India.
\end{document}